\documentclass[11pt]{amsart}

\usepackage[a4paper,hmargin={2.5cm,2.5cm},vmargin={2.4cm,2.4cm},heightrounded,marginparwidth=2.3cm, marginparsep=0.1cm]{geometry}

\usepackage[utf8]{inputenc}
\usepackage[T1]{fontenc}

\usepackage{amssymb,amscd,amsthm,latexsym}

\usepackage[british]{babel}
\usepackage[all]{xy}
\usepackage{tikz-cd}
\RequirePackage[shortlabels]{enumitem}

\theoremstyle{plain}
\newtheorem{theorem}[subsection]{Theorem}
\newtheorem{lemma}[subsection]{Lemma}
\newtheorem{proposition}[subsection]{Proposition}
\newtheorem{corollary}[subsection]{Corollary}

\theoremstyle{definition}
\newtheorem{definition}[subsection]{Definition}

\theoremstyle{remark}

\newtheorem{remark}[subsection]{Remark}

\newcommand{\C}{\ensuremath{\mathsf{C}}}
\newcommand{\A}{\ensuremath{\mathsf{A}}}
\newcommand{\B}{\ensuremath{\mathsf{B}}}
\newcommand{\D}{\ensuremath{\mathsf{D}}}

\newcommand{\Frm}{\ensuremath{\mathsf{Frm}}}

\newcommand{\Ord}{\ensuremath{\mathsf{Ord}}}
\newcommand{\Inf}{\ensuremath{\mathsf{Inf}}}
\newcommand{\Top}{\ensuremath{\mathsf{Top}}}
\newcommand{\CH}{\ensuremath{\mathsf{CompHaus}}}
\newcommand{\Cls}{\ensuremath{\mathsf{Cls}}}
\newcommand{\App}{\ensuremath{\mathsf{App}}}
\newcommand{\Set}{\ensuremath{\mathsf{Set}}}
\newcommand{\ASet}{\ensuremath{\mathsf{AfSet}(A)}}
\newcommand{\AASet}{\ensuremath{\mathsf{AfAlgSet}(A)}}
\newcommand{\VSet}{\ensuremath{\mathsf{AfSet}}(V)}

\newcommand{\VCat}{V\mbox{-}\ensuremath{\mathsf{Cat}}}
\newcommand{\Vccd}{V\mbox{-}\ensuremath{\mathsf{ccd}}}
\newcommand{\UVCat}{(U,V)\mbox{-}\ensuremath{\mathsf{Cat}}}

\newcommand{\FcAlgA}{\mathsf{FcAlg}(A)}

\def\mathrmdef#1{\expandafter\def\csname#1\endcsname{{\rm#1}}}
\mathrmdef{co}\mathrmdef{lan}\mathrmdef{ran}\mathrmdef{ev}\mathrmdef{op}\mathrmdef{id}

\newcommand{\yoneda}{y}
\newcommand{\two}{\mathsf{2}}

\usepackage{rotating}
\newcommand\downarrowtail{%
  \mathrel{\vcenter{\hbox{\rotatebox{-90}{\(\rightarrowtail\)}}}}}

\newcommand{\pararrows}[4]{%
  \begin{tikzcd}[ampersand replacement=\&]
    #1\ar[shift left]{r}{#3}\ar[shift right]{r}[swap]{#4} \& #2
  \end{tikzcd}
}

\counterwithin*{equation}{section}

\usepackage[hypertexnames=false,colorlinks=true]{hyperref}

\begin{document}

\title{A variety of co-quasivarieties}

\author{Maria Manuel Clementino}
\thanks{Partially supported by \textit{Centro de Matemática da Universidade de Coimbra} (CMUC), funded by the Portuguese Government through FCT/MCTES, DOI 10.54499/UIDB/00324/2020.}
\address{University of Coimbra, CMUC, Department of Mathematics, 3000-143 Coimbra, Portugal}
\email{mmc@mat.uc.pt}

\author{Carlos Fitas}
\thanks{Partially supported by \textit{Centro de Matemática da Universidade de Coimbra} (CMUC), funded by the Portuguese Government through FCT/MCTES, DOI 10.54499/UIDB/00324/2020, and the FCT Ph.D. grant SFRH/BD/150460/2019.}
\address{University of Coimbra, CMUC, Department of Mathematics, 3000-143 Coimbra, Portugal}
\email{cmafitas@gmail.com}

\author{Dirk Hofmann}
\thanks{Partially supported by the Center for Research and Development in Mathematics and Applications (CIDMA) through the Portuguese Foundation for Science and Technology (FCT -- Fundação para a Ciência e a Tecnologia), references UIDB/04106/2020 and UIDP/04106/2020.}
\address{Center for Research and Development in Mathematics and Applications, Department of Mathematics, University of Aveiro, Portugal}
\email{dirk@ua.pt}

\keywords{Quasivariety, topological category, affine set}

\subjclass{18C10, 08C15, 18C05, 08A65, 18D20, 54B30}

\begin{abstract}
  It is shown that the duals of several categories of topological flavour, like the categories of ordered sets, generalised metric spaces, probabilistic metric spaces, topological spaces, approach spaces, are quasivarieties, presenting a common proof for all such results.
\end{abstract}

\maketitle

\section*{Introduction}

It is known since the work of Manes \cite{Man69} that the category \(\CH\) of compact Hausdorff spaces and continuous maps, traditionally belonging to the realm of topology, can be also understood as a variety of algebras (without rank). Approximately at the same time, Duskin \cite{Dus69} pointed out that also the dual category of \(\CH\) is a variety: by Urysohn's Lemma, the unit interval \([0,1]\) is a regular injective regular cogenerator of \(\CH\); moreover \(\CH^{\op}\) is exact (see \cite[Corollary~1.11 of Chapter~9]{BW85}). On the other hand, the category \(\Top\) of topological spaces and continuous maps is probably ``as far as one can get'' from algebra, so it might come ``as some surprise \dots\ that the situation is quite different when it comes to the dual category'' \cite{BP95}. In fact, once suspected, it is not hard to see from an abstract point of view that \(\Top^{\op}\) is a quasivariety: \(\Top\) is (co)complete, and the space \(3\) --- an amalgamation of the Sierpinski space and the two-element indiscrete space --- is a regular injective regular cogenerator of \(\Top\). It is also noted that \(\Top^{\op}\) is not exact. In concrete terms, \cite{BP95} gives an explicit description of the corresponding algebraic theory. Later, in \cite{AP97}, Adámek and Pedicchio described $\Top^\op$ as a regular epireflective subcategory of the variety of topological systems. Later, in \cite{PW99} Pedicchio and Wood pointed out that analogous arguments show that the dual of the category $\Ord$ of (pre)ordered sets and monotone maps is a quasivariety. Here we use the ideas of Adámek and Pedicchio \cite{AP97}, namely the equivalence between the category of topological systems and a special comma category, and results of Barr \cite{Bar72}, to give a common proof that a variety of categories of topological nature have quasivarieties as duals.

\section{The global construction}\label{sect:1}

We recall that a \emph{varietal category}, or just a \emph{variety}, is a category monadic over $\Set$, while a \emph{quasivariety} is a regular epireflective (full) subcategory of a variety, or, equivalently, a (full) subcategory closed under subobjects and products. The goal of this section is to identify certain comma categories of functors between varieties as (two-sorted) varieties. We start with the following observation.

\begin{lemma}\label{d:lem:1}
  Let \(I \colon\B\to\A\) be a functor. Then the functor
  \begin{displaymath}
    F \colon \A\mathrel{\downarrow}I \longrightarrow \A\times\B
  \end{displaymath}
  sending \(A\to IB\) to \((A,B)\) and acting as identity on morphisms has a right adjoint provided that \(\A\) has binary products. If \(I\) has a left adjoint and \(\A\) has binary coproducts, then \(F\) has a left adjoint.
\end{lemma}
\begin{proof}
  The right adjoint $R$ of $F$ is defined by $R(A,B)=(\xymatrix{A\times IB\ar[r]^-{\pi_2}&IB})$, and $R(f,h)=(f\times Ih,h)$, which clearly makes the diagram
  \[
    \xymatrix{A\times IB\ar[d]_{f\times Ih}\ar[r]^-{\pi_2}&IB\ar[d]^{Ih}\\
      A'\times IB'\ar[r]_-{\pi_2}&IB'}
  \]
  commute. Again, it is straightforward to show that $F\dashv R$, with unit $\gamma$ given by $\gamma_g=(\langle 1_A,g\rangle,1_B)$ as in the diagram
  \[
    \xymatrix{A\ar[r]^-g\ar[d]_{\langle 1_A,g\rangle}&IB\ar[d]^{I1_B}\\
      A\times IB\ar[r]_-{\pi_2}&IB.}
  \]
  Assume now that \(I\) has a left adjoint $J\colon \A\to\B$, with unit $\eta$. Then define
  \[
    L(A,B)
    =
    (A\xrightarrow{\;\eta_{A}\;}IJA\xrightarrow{\;I\iota_{JA}\;}I(JA+B))
  \]
  where $\iota_{JA}\colon JA\to JA+B$ is the coproduct projection, and define $L(f,h)=(f,Jf+h)$ as in the commutative diagram
  \[
    \xymatrix{A\ar[r]^-{\eta_A}\ar[d]_f
      &IJA\ar[d]^{IJf}\ar[r]^-{I\iota_{JA}}&I(JA+B)\ar[d]^{I(Jf+h)}\\
      A'\ar[r]_-{\eta_{A'}}&IJA'\ar[r]_-{I\iota_{JA'}}&I(JA'+B')}
  \]
  It is straightforward to check that $L$ is left adjoint to $F$, with unit $\rho$ given by $\rho(A,B)=(1_A,\iota_B)\colon(A,B)\to(A,JA+B)$, where $\iota_B\colon B\to JA+B$ is the coproduct coprojection.
\end{proof}

Assume now that \(\A\) and \(\B\) are varieties, with forgetful functor \(G \colon\A\to\Set\) and \(H \colon\B\to\Set\), respectively. Then also \(G\times H \colon\A\times\B\to\Set\times\Set\) is monadic, and so is \(F \colon \A\mathrel{\downarrow}I\to\A\times\B\), for every right adjoint functor \(I \colon\B\to\A\). In fact, \(\A\mathrel{\downarrow}I\) is cocomplete, \(F\) is left and right adjoint and therefore preserves in particular coequalizers, and one easily verifies that \(F\) reflects isomorphisms. We wish to conclude that then also \(\A\mathrel{\downarrow}I\) is monadic over \(\Set\times\Set\); however, monadic functors are not stable under composition in general. In \cite{Bar72} Barr presented a criterion to identify monadic functors as composite of functors with suitable properties which we recall next.

\begin{definition}
  A parallel pair of morphisms \(\pararrows{X}{Y}{f}{g}\) in a category $\C$ is said to be a \emph{split coequalizer pair} (or simply a \emph{split pair}) if there exists an object $Z$ in $\C$ and morphisms $h\colon Y\to Z$, $k\colon Z\to Y$, $s\colon Y\to X$
  \begin{displaymath}
    \begin{tikzcd}
      X\ar[shift left]{r}{f}\ar[shift right]{r}[swap]{g}
      & Y\ar[shift right]{r}[swap]{h}
      \ar[bend left=50]{l}{s}
      & Z\ar[shift right]{l}[swap]{k}
    \end{tikzcd}
  \end{displaymath}
  such that $hf=hg$, $gs=1_Y$, $fs=kh$.

  If $F\colon\C\to\D$ is a functor, \(\pararrows{X}{Y}{f}{g}\) is said to be an \emph{$F$-split pair} if \(\pararrows{FX}{FY}{Ff}{Fg}\) is a split pair in $\D$.
\end{definition}

\begin{proposition}[\cite{Bar72}]
  \label{d:prop:1}
  A functor \((A\xrightarrow{\;U\;}D)=(A\xrightarrow{\;F\;}B\xrightarrow{\;G\;}C\xrightarrow{\;H\;}D)\) is a monadic functor provided that:
  \begin{enumerate}
  \item $\A$ has coequalizers of $U$-split pairs;
  \item $H$ creates limits and every $H$-split pair is a split pair;
  \item $G$ creates limits, and coequalizers of $G$-split pairs exist in $\B$ and are preserved by $G$;
  \item $F$ creates limits and preserves all coequalizers;
  \item $HGF$ has a left adjoint.
  \end{enumerate}
\end{proposition}

\begin{corollary}
  \label{d:cor:1}
  Let \(G \colon\A\to\Set\) and \(H \colon\B\to\Set\) be monadic functors and \(I \colon\B\to\A\) be a right adjoint functor. Then the composite functor
  \begin{displaymath}
    \A\mathrel{\downarrow}I
    \xrightarrow{\quad F\quad}
    \A\times\B
    \xrightarrow{\quad G\times H\quad}
    \Set\times\Set
  \end{displaymath}
  is monadic.
\end{corollary}
\begin{proof}
  First note that with \(G \colon\A\to\Set\) and \(H \colon\B\to\Set\) also \(G\times H \colon\A\times\B\to\Set\times\Set\) is monadic. As $\A$ and $\B$ are cocomplete, $\A\downarrow I$ is a cocomplete category. By Lemma~\ref{d:lem:1}, the functor $F\colon\A\downarrow I\to\A\times B$ has both a left and a right adjoint. Therefore $F$ preserves both limits and colimits. Since \(I\) preserves limits, one easily verifies that \(F\) creates limits. Therefore the assertion follows from Proposition~\ref{d:prop:1}.
\end{proof}

\begin{remark}
  \label{d:rem:1}
  Recall that every concrete functor \(I \colon\B\to\A\) between varieties has a left adjoint. Hence, in this situation, we can think of \(\A\mathrel{\downarrow}I\) as a category of two sorted algebras: with sorts \(A\) and \(B\), with the algebraic theory of \(G \colon\A\to\Set\) at sort \(A\) and of \(H \colon\B\to\Set\) at sort \(B\), and an operation of type \(A\to B\) together with the equations identifying this operation as a homomorphism.
\end{remark}

\begin{proposition}\label{prop}
  If $\A$ and $\B$ are varieties and $I\colon\B\to\A$ is a right adjoint functor, then the full subcategory $\A\downarrowtail I$ of the comma category $\A\downarrow I$ consisting of those objects \(A\xrightarrow{\;g\;} IB\) with $g$ monic is a regular epireflective subcategory of $\A\downarrow I$.
\end{proposition}
\begin{proof}
  First of all we note that, as a variety, $\A$ admits the (regular epi, mono)-factorization system. For each object \(A\xrightarrow{\;g\;}IB\), let \(A\xrightarrow{\;g\;}IB=A\xrightarrow{\;e\;}X\xrightarrow{\;m\;}IB\) be its (regular epi,mono)-factorization. Then $(e,1_B)$ is the regular epireflection of $A$ in $\A\downarrowtail I$: indeed, given a morphism $(f,h)\colon(A\xrightarrow{\;g\;}IB)\to(X'\xrightarrow{m'}IB')$ with $m'$ monic, commutativity of the diagram
  \begin{displaymath}
    \begin{tikzcd}
      A %
      \ar{rr}{g} %
      \ar{rd}[swap]{e} \ar{dd}[swap]{f} %
      && IB %
      \ar{dd}{Ih} \\
      & X\ar{ru}[swap]{m}\ar[dotted]{dl}{d}\\
      X' %
      \ar{rr}[swap]{m'} %
      && IB' %
    \end{tikzcd}
  \end{displaymath}
  gives, by orthogonality of the factorization system, a unique morphism $d\colon X\to X'$ such that $m'\cdot d=Ih\cdot m$ and $d\cdot e=f$. This yields a (unique) morphism $(d,h)\colon(X\xrightarrow{\;m\;}IB)\to(X'\xrightarrow{\;m'\;}IB')$ so that $(d,h)\cdot (e,1_B)=(f,h)$ as required.
\end{proof}

In the situation of Remark~\ref{d:rem:1}, the category $\A\downarrowtail I$ can be described as the category of those algebras \((g \colon A\to IB)\) of \(\A\downarrow I\) satisfying the implication
\begin{displaymath}
  \forall x,y\,.\, (g(x)=g(y)\implies x=y).
\end{displaymath}

The main point to make here is that, making again use of \cite{Bar72}, $\A\downarrowtail I$ can be also presented as a quasivariety over \(\Set\).

\begin{theorem}
  \label{th}
  If $\A$ and $\B$ are varieties admitting constants and $I\colon\B\to\A$ is a right adjoint functor, then the comma category $\A\downarrow I$ is monadic over $\Set$.
\end{theorem}
\begin{proof}
  Similarly to Corollary~\ref{d:cor:1}, we obtain immediately that the composite
  \begin{displaymath}
    \A\downarrow I\xrightarrow{\quad F\quad}
    \A\times B\xrightarrow{\quad G\times H\quad}
    (\Set\times\Set)^{*}
  \end{displaymath}
  is monadic, where
  \begin{itemize}
  \item $F(A\to IB)=(A,B)$ and $F(f,h)=(f,h)$;
  \item $G\colon\A\to\Set$ and $H\colon\B\to\Set$ are monadic functors, and $G\times H$ is their product restricted to the full subcategory of $\Set\times\Set$ consisting of pairs $(X,Y)$ of sets where either both are empty or both are non-empty.
  \end{itemize}
  Consider now also the product functor
  \begin{displaymath}
    (\Set\times\Set)^{*}\xrightarrow{\quad P\quad}\Set.
  \end{displaymath}
  The functor $P$ has a left adjoint, namely the diagonal functor $\Set\to(\Set\times\Set)^*$. Furthermore, that $P$ creates limits and every $P$-split pair is a split pair follows from \cite[Theorem 2]{Bar72}. Therefore, by Proposition~\ref{d:prop:1}, the composite
  \begin{displaymath}
    \A\downarrow I\xrightarrow{\quad F\quad}
    \A\times B\xrightarrow{\quad G\times H\quad}
    (\Set\times\Set)^{*}\xrightarrow{\quad P\quad}\Set
  \end{displaymath}
  is monadic.
\end{proof}

Hence, together with Proposition~\ref{prop}, we obtain

\begin{theorem}
  If $\A$ and $\B$ are varieties admitting constants and $I\colon\B\to\A$ is a right adjoint functor, then $\A\downarrowtail I$ is a quasivariety over \(\Set\).
\end{theorem}

In the remaining sections we will show that the results of this section can be used to deduce that several categories -- dual to categories of topological nature -- are quasivarietal.

\section{Diers' affine algebraic sets}

Motivated by questions in geometry, the notion of affine set was introduced by Diers in \cite{Die96, Die99}. Among other results, the author proves that the category of affine sets is topological over \(\Set\), and that, moreover, many categories studied in topology are instances of this notion. Using the results of Section \ref{sect:1}, it is easy to deduce that the dual of the category of affine sets is a quasivariety, as we explain next.

Given a variety $\A$ and an object $A$ of $\A$, an \emph{affine set} is a set $X$ equipped with a subalgebra $S$ of the algebra $A^X$. A morphism of affine sets $f\colon(X,S)\to(Y,T)$ is a map $f\colon X\to Y$ such that $A^f\colon A^Y\to A^X$ (co)restricts to $T\to S$.

The category \(\ASet\) of affine sets and their morphisms admits a canonical closure operator, called \emph{Zariski-closure}: For an affine set \((X,S)\) and \(M\subseteq X\), the Zariski-closure of \(M\) is given by
\begin{displaymath}
  \overline{M}=\{\mathrm{Eq}(\varphi,\psi)\mid \varphi,\psi\in S, M\subseteq \mathrm{Eq}(\varphi,\psi)\}.
\end{displaymath}
An affine set \((X,S)\) is called \emph{separated} whenever the cone \(S\) is point-separating, that is, for all \(x\neq y\) in \(X\) there is some \(\varphi\in S\) with \(\varphi(x)\neq\varphi(y)\), and it is called \emph{algebraic} if it is separated and Zariski-closed in every separated affine set. The full subcategory of \(\ASet\) defined by all algebraic affine sets is denoted by \(\AASet\). It is shown in \cite{Die96, Die99} that \(\AASet\) is dually equivalent to the category of \emph{functional algebras} over \(A\),
\begin{displaymath}
  \AASet^{\op}\simeq\FcAlgA;
\end{displaymath}
here \(\FcAlgA\) is the full subcategory of \(\A\) defined by those algebras which are subalgebras of powers of \(A\). The category \(\FcAlgA\) is a regular-epireflective subcategory of \(\A\), hence a quasivariety over \(\Set\), which identifies \(\AASet^{\op}\) as a quasivariety.

Using the results of Section \ref{sect:1} it is easy to deduce that also the dual of \(\ASet\) is a (two-sorted) quasivariety. Indeed:
\begin{enumerate}[(a)]
\item The functor $I\colon \Set^\op\to\A$, which assigns to each $X$ the algebra $A^X$ and to each map $f\colon X\to Y$ the morphism $A^f\colon A^Y\to A^X$ with $A^f(g)=g\cdot f$, has as left adjoint the hom-functor $J=\A(-,A)\colon\A\to\Set^\op$.
\item Both $\A$ and $\Set^\op$ are monadic over $\Set$, with respect to the forgetful functor \(\A\to\Set\) and \(\Set(-,2) \colon\Set^{\op}\to\Set\) \cite{Sob92}, respectively.
\item $\ASet$ is dually equivalent to $\A\downarrowtail I$: each object $(X,S)$ can be seen as a monomorphism $SX\hookrightarrow IX$, and a morphism $f\colon(Y,T)\to(X,S)$ as a morphism in $\A\downarrowtail I$:
  \begin{displaymath}
    \begin{tikzcd} 
      S %
      \ar{r}{} %
      \ar{d}[swap]{} %
      & I X %
      \ar{d}{I(f)} \\
      T %
      \ar{r}[swap]{} %
      & IY %
    \end{tikzcd}
  \end{displaymath}
  Conversely, to every monomorphism $m\colon B\to A^X$ one may assign its image $m(B)\hookrightarrow A^X$ and to every morphism $(f,h)$ as in the diagram
  \begin{displaymath}
    \begin{tikzcd} 
      B %
      \ar{r}{m} %
      \ar{d}[swap]{f} %
      & A^{X} %
      \ar{d}{I(h)} \\
      C %
      \ar{r}[swap]{n} %
      & A^{Y} %
    \end{tikzcd}
  \end{displaymath}
  one may assign the morphism $h\colon(Y,n(C))\to(X,m(B))$; since $m$ and $n$ are monomorphisms, $f$ is exactly the (co)restriction of $I(h)$.
\end{enumerate}
Using Proposition \ref{prop} and Theorem \ref{th} we can then conclude that:
\begin{proposition}
  Given a variety $\A$ with constants and an object of $\A$, $\ASet^\op$ is a quasi\-variety over \(\Set\).
\end{proposition}

\section{Examples}
\label{sec:examples}

\subsection{Quantale-enriched categories}
\label{sec:quant-enrich-categ}

We recall that a complete ordered set is said to be constructively completely distributive (ccd for short) if the monotone map $\bigvee\colon \mathbb{D}X\to X$, where $\mathbb{D}X$ is the lattice of downsets of $X$, has a left adjoint \cite{FW90}. In the more general context of $V$-categories, for a commutative quantale \(V\), a $V$-category $(X,a)$ is cocomplete if and only if its Yoneda embedding $\yoneda_X\colon X\to PX$ admits a left adjoint $\bigvee_X\colon PX\to X$. As for the case $V=\two$, when $\bigvee_X$ admits a left adjoint $X$ is said to be \emph{constructively completely distributive}, or just ccd. For more information on (completely distributive) \(V\)-categories we refer to \cite{Stu05, Stu06}. We point out that the notion of complete distributivity in this context was introduced in \cite{Stu07} under the designation \emph{totally continuous}. The category $\Vccd$ has as objects the ccd $V$-categories and as morphisms the $V$-functors which preserve both weighted limits and weighted colimits. As shown in \cite{PZ15}, $\Vccd$ is a variety with respect to the canonical forgetful functor \(\Vccd\to\Set\), which in our study above will play the role of the variety $\A$. The role of the algebra $A$ is played by the ccd $V$-category $(V,\hom)$. That is, our functor $I\colon\Set^\op\to\Vccd$ is given by $IX=V^X$. So, in order to conclude that $(\VCat)^\op$ is a quasivariety, it remains to show that \(\VCat\) is equivalent to \(\VSet\).

\begin{proposition}
  The category \(\VCat\) is concretely isomorphic to the category \(\VSet\).
\end{proposition}
\begin{proof}
  We define functors
  \begin{displaymath}
    F \colon\VCat\to\VSet
    \qquad\text{and}\qquad
    G \colon\VSet\to\VCat
  \end{displaymath}
  where \(F\) sends the \(V\)-category \(X\) to the affine set \((X,\VCat(X,V))\), \(G\) sends the affine set \((X,S)\) to the \(V\)-category \(X\) equipped with the initial structure with respect to \(S\), and both functors act identically on morphisms. Since \(V\) is initially dense in \(\VCat\), the composite functor \(GF\) is the identity. Now let \((X,S)\) be an affine set and let \((X,a)\) be the initial \(V\)-category with respect to \(S\). Clearly, \(S\subseteq\VCat(X,V)\). For every \(x\in X\),
  \begin{displaymath}
    a(x,-)=\bigwedge_{\varphi\in S}\hom(\varphi(x),\varphi(-)),
  \end{displaymath}
  hence \(a(x,-)\) belongs to \(S\). If \(\psi \colon(X,a)\to(V,\hom)\) is a \(V\)-functor, then, since
  \begin{displaymath}
    \psi(-)=\bigvee_{y\in X}\psi(y)\otimes a(y,-),
  \end{displaymath}
  \(\psi\in S\), and we conclude that \(S=\VCat(X,V)\). Therefore \(FG\) is the identity functor.
\end{proof}

\begin{proposition}
  Under the equivalence above, an affine set is algebraic if and only if the corresponding \(V\)-category is separated and Cauchy-complete. Hence, the full subcategory \(\VCat_{\mathrm{sep},\mathrm{cc}}\) of \(\VCat\) defined by all separated and Cauchy-complete \(V\)-categories is dually equivalent to a quasivariety over \(\Set\).
\end{proposition}
\begin{proof}
  By \cite{HT10}, the Zariski-closure corresponds precisely to the L-closure.
\end{proof}

\subsection{Topological spaces}
\label{sec:topologial-spaces}

The category \(\Top\) of topological spaces and continuous maps is isomorphic to the category \(\mathsf{AfSet}(2)\), for \(\A\) the variety \(\Frm\) of frames and frame homomorphisms and the two-element frame \(2\). The affine algebraic frames correspond precisely to the sober spaces. We note that, equivalently, one might consider here the variety \(\A=\mathsf{coFrm}\) of co-frames and homomorphisms.

The corresponding description of \(\Top^{\op}\) as the two-sorted quasivariety \(\Frm\downarrowtail I\), for the full embedding \(I \colon \mathsf{CABool}\to\Frm\), is presented in \cite{AP97} and constitutes one of the main motivation for this note. Furthermore, it is shown in \cite{AP97} that the slice category \(\Frm\downarrow I\) is equivalent to the category of grids and homomorphisms, linking this way their approach to the one taken in \cite{BP95}.

\subsection{Closure spaces}
\label{sec:closure-spaces}

Let \(\A\) be the category \(\Inf\) of complete lattices and infima-preserving maps, and consider the two-element complete lattice \(2\). Then \(\mathsf{AfSet}(2)\) is isomorphic to the category \(\Cls\) of closure spaces and continuous maps \cite{Die96}. Hence, \(\Cls^{\op}\) is a quasivariety over \(\Set\).

\subsection{Approach spaces}

The category \(\App\) of \emph{approach spaces} and contraction maps was introduced by R.~Lowen in \cite{Low89} (see \cite{Low97} for more details) as a common generalisation of topological spaces and metric spaces. Similar to topological spaces, approach spaces can be equivalently described by convergence, neighborhood systems, and ``metric variants'' of closed (respectively open) sets. As pointed out in \cite{CGL10}, the description via \emph{regular function frames} identifies approach spaces as affine sets for the algebraic theory of \emph{approach frames} (introduced in \cite{BLO06} and further studied in \cite{Olm05, OV10}) and the approach frame \([0,\infty]\); in other words, \(\App\) is concretely isomorphic to the category \(\mathsf{AfSet}([0,\infty])\). Consequently, in analogy to the case of topological spaces, also \(\App^{\op}\) is a quasivariety. It is also shown in \cite{CGL10} that in this context the affine algebraic sets correspond precisely to the sober approach spaces.

Similarly to the case of \(\Top\), the category \(\App^{\op}\) is not exact.

\subsection{A \L{}ukasiewicz variant of approach spaces}

It is shown in \cite{CH03} that the category \(\App\) fits into the framework of \emph{monoidal topology} \cite{HST14}: \(\App\) is concretely isomorphic to the category \(\UVCat\) of \((U,V)\)-categories and \((U,V)\)-functors, for the ultrafilter monad \(U\) on \(\Set\) and the quantale \(V=[0,\infty]\) with monoidal structure given by addition. This begs the question whether the method presented in this note applies to \(\UVCat\) for other quantales \(V\) as well. At this moment we do not have a general answer, but at least for the quantale \(V=[0,1]\) with the \L{}ukasiewicz sum defined by \(u\otimes v=\max\{u+v-1,0\}\) the answer is positive. To explain this, we observe first that the category of complete and finitely cocomplete \(V\)-categories and \(V\)-functors preserving limits and finite colimits can be described by operations and equations: this category is concretely isomorphic to the category of algebras and homomorphisms for a class \(\Omega\) of operation symbols and a class \(E\) of equations (similar to \cite{PT89} and \cite[Remark~2.10]{HN18}). We then take \(\A\) to be the subvariety generated by \([0,1]\), that is, the full subcategory containing precisely the homomorphic image of subobjects of powers of \([0,1]\). By \cite[Corollary~3.17]{HN23}, we have:
\begin{proposition}
  For the quantale \(V=[0,1]\) with the \L{}ukasiewicz sum, the category \(\UVCat\) is concretely isomorphic to the category \(\mathsf{AfSet}([0,1])\), with respect to the variety \(\A\) and the \(V\)-category \(V=[0,1]\). Consequently, \((\UVCat)^{\op}\) is a quasivariety.
\end{proposition}



\end{document}